\documentclass[12pt,conference]{IEEEtran}
\ifCLASSOPTIONcompsoc
  \usepackage[nocompress]{cite}
\else
  \usepackage{cite}
\fi

\usepackage{tikz-cd}
\usepackage{etex}
\usepackage{amsmath,amssymb,url}
\usepackage{fdsymbol}
\usepackage{bussproofs}
\usepackage{epic}
\usepackage{eufrak}
\usepackage{eepic}
\newsavebox{\wwide}
\newcommand{\wwidehat}[1]{\sbox{\wwide}{$#1$}
\ifdim\wd\wwide < 1.1 em \widehat{#1} \else
\setlength
{\unitlength}{0.01\wd\wwide}\overset
{\begin{picture}(100,6)
\path(0,0)(50,6)(100,0)
\end{picture}}{#1}\fi}

\newcommand{\wwidetilde}[1]{\sbox{\wwide}{$#1$}
\ifdim\wd\wwide < 1.1 em \widetilde{#1} \else
\setlength
{\unitlength}{0.01\wd\wwide}\overset
{\begin{picture}(100,6)
\path(0,0)(33,6)(45,6)(55,0)(67,0)(100,6)
\end{picture}}{#1}\fi}

\definecolor{ChadDarkBlue}{rgb}{.1,0,.2}  
\definecolor{ChadBlue}{rgb}{.1,.1,.5}  
\definecolor{ChadRoyal}{rgb}{.2,.2,.8}  
\definecolor{ChadGreen}{rgb}{0,.4,0}    
\definecolor{ChadRed}{rgb}{.5,0,.5}  
\usepackage{amsthm}
\usepackage{epic,eepic}
\usepackage[mathscr]{euscript}
\usepackage{enumerate}

\usepackage{diagrams}
\usepackage{graphicx}

\usepackage{lineno}
\usepackage{comment}

\usepackage{tikz}
\usetikzlibrary{graphs}
\usetikzlibrary{positioning}

\def\smallskip{\vskip\smallskipamount}
\def\medskip{\vskip\medskipamount}
\def\bigskip{\vskip\bigskipamount}

\numberwithin{equation}{section}

\theoremstyle{plain}
\newtheorem{theorem}{Theorem}[section]

\newtheorem{lemma}[theorem]{Lemma}
\newtheorem{proposition}[theorem]{Proposition}

\newtheorem{corollary}[theorem]{Corollary}

\newtheorem{defn}[theorem]{Definition}

\theoremstyle{definition}

\newtheorem{remark}[theorem]{Remark}

\DeclareTextCommand{\elqq}{T1}{\leavevmode\char16\nobreak\hskip0pt}
\DeclareTextCommand{\erqq}{T1}{{\edef\@SF{\spacefactor\the\spacefactor}%
\nobreak\char17\@SF\relax}}

\newcommand{\cat}{\mathbf}
\newcommand{\ra}{\rightarrow}
\newcommand{\la}{\leftarrow}

\newcommand{\da}{\downarrow\!}

\newcounter{ok}
{\end{list}}

\newcounter{aok}
{\end{list}}

\def\go#1;#2;#3 {\vbox to0pt{\kern-#3\hbox{\kern#2 #1}\vss}\nointerlineskip}

\newcommand{\structA}{\mathcal A}
\newcommand{\map}[3]{#1 \colon #2 \to #3}

\newcommand{\Qsup}{\bigsqcup}

\renewcommand{\sup}{\bigvee}

\mathchardef\mhyphen="2D

\newcommand{\catQSupOAlg}{Q\text{-}\mathbf{Sup}\text{-}\Omega\text{-}\mathbf{Alg}}
\newcommand{\catQModOAlg}{Q\text{-}\mathbf{Mod}\text{-}\Omega\text{-}\mathbf{Alg}}
\newcommand{\catOAlg}{\Omega\text{-}\mathbf{Alg}}

 \newcommand{\N}{{\mathbb{N}}}

 \renewcommand{\le}{\leqslant}

\begin{document}
\title{A representation theorem for quantale valued sup-algebras}
\author{%
        \IEEEauthorblockN{Jan~Paseka}
\IEEEauthorblockA{Department of Mathematics and Statistics\\
		Faculty of Science, Masaryk University\\
		{Kotl\'a\v r{}sk\' a\ 2}, CZ-611~37~Brno, Czech Republic\\
        E-mail: paseka@math.muni.cz}
\and 
\IEEEauthorblockN{Radek~{\v S}lesinger}
\IEEEauthorblockA{Department of Mathematics and Statistics\\
		Faculty of Science, Masaryk University\\
		{Kotl\'a\v r{}sk\' a\ 2}, CZ-611~37~Brno, Czech Republic\\
        E-mail: xslesing@math.muni.cz}}



\markboth{A representation theorem for quantale valued sup-algebras}%
{A representation theorem for quantale valued sup-algebras}
%

%




\IEEEcompsoctitleabstractindextext{%
\begin{abstract}
With this paper we hope to contribute to the theory of quantales and quantale-like structures. It considers
the notion of $Q$-sup-algebra and shows a representation theorem for such
structures generalizing the well-known representation theorems for quantales and sup-algebras. In addition, we present some important properties of the category of 
$Q$-sup-algebras.
\end{abstract}
\begin{IEEEkeywords}  sup-lattice, sup-algebra, quantale, $Q$-module, $Q$-order,  $Q$-sup-lattice, 
 $Q$-sup-algebra.
\end{IEEEkeywords}}
      
\maketitle  

\IEEEdisplaynotcompsoctitleabstractindextext

%

{\section*{Introduction}}

\label{intro}
\IEEEPARstart{T}{wo} equivalent structures, quantale modules \cite{abramsky-vickers} and $Q$-sup-lattices \cite{zhang-xie-fan} were independently introduced and studied. Stubbe \cite{stubbe-tcc} constructed an isomorphism between the categories
of right $Q$-modules and cocomplete skeletal $Q$-categories for a given
unital quantale $Q$. Employing his results,  Solovyov \cite{solovyov-qa} obtained 
an isomorphism between the categories of $Q$-algebras and $Q$-valued quantales, 
where $Q$ is additionally assumed to be commutative. 

Resende introduced (many-sorted) sup-algebras that are certain partially 
ordered algebraic structures which generalize quantales, frames and biframes (pointless
topologies) as well as various lattices of multiplicative 
ideals from ring theory and functional analysis
(C*-algebras, von Neumann algebras).  One-sorted case was studied e.g. by 
Zhang and Laan \cite{zhang-laan},  Paseka \cite{paseka}, and, in the generalized form 
of $Q$-sup-algebras by {\v S}lesinger \cite{slesinger2, slesinger}.

The paper is organized as follows. First we present 
several necessary algebraic concepts as sup-lattice, sup-algebra, quantale and quantale module.
Using quantales as a base structure for valuating fuzzy concepts, we recall the notion of a $Q$-order -- a fuzzified
variant of partial order relations, and a $Q$-sup-lattice for a fixed unital commutative quantale $Q$.




We then recall 
Solovyov's isomorphism between the category of $Q$-sup-lattices and the category of $Q$-modules, 
and between the category of $Q$-sup-algebras and the category of $Q$-module-algebras. 
This isomorphism provides a relation between quantale-valued sup-algebras, 
which are expressed through fuzzy concepts, and quantale module-algebras, which is a notion
expressed in terms of universal algebra.

In Section \ref{sectimpprop} we establish several important properties of the category of $Q$-sup-algebras, e.g., 
an adjoint situation between categories of  $Q$-sup-algebras and $Q$-algebras, and the fact that 
the category of $Q$-sup-algebras is a monadic construct. 

In Section \ref{nucleiprop} we focus on the notion of a $Q$-MA-nucleus for $Q$-module-algebras and state its 
main properties.  In the last section we introduce a $Q$-MA-nucleus on the free Q-sup-algebra and using it, we 
establish our main  theorem for  $Q$-sup-algebras  that 
generalizes the well-known representation theorems for quantales and sup-algebras. 

In this paper, we take for granted the concepts and results on quantales, category theory 
and universal algebra. To obtain more information on these topics, we direct the reader
to  \cite{adamek}, \cite{kruml} and \cite{rosenthal1}. 

\medskip
\section{Basic notions, definitions and results}

\subsection{Sup-lattices, sup-algebras, quantales and quantale modules}

A {\em\bfseries sup-lattice} $A$ is a partially ordered set 
(complete lattice) in which every subset 
$S$ has a join (supremum) $\bigvee S$, and therefore 
also a meet (infimum) $\bigwedge S$. The greatest element 
 is denoted by $\top$, the least element by $\bot$. 
A {\em\bfseries sup-lattice homomorphism} $f$ between sup-lattices 
$A$ and $B$ is a join-preserving mapping from $A$ to $B$, 
i.e., 
$f \left(\sup S \right) = \sup \{f(s) \mid s \in S \}$  for every subset 
$S$  of $A$. The category of sup-lattices will be denoted  $\cat{Sup}$.
Note that a mapping $f\colon{}A\to B$ if a sup-lattice homomorphism if and only 
if it has a {\em\bfseries right adjoint} $g \colon B\to A$, 
by which is meant a mapping $g$ that satisfies
\[
f(a)\le b\iff a\le g(b)
\]
for all $a\in A$ and $b\in B$. We write $f\dashv g$ in order to state 
that $g$ is a right adjoint to $f$ (equivalently, $f$ is a left adjoint to $g$). 

{A \emph{type} is a set $\Omega$ of \emph{function symbols}. To 
each $\omega \in \Omega$, a number $n \in \N_0$ is assigned, which is called 
the \emph{arity} of $\omega$ (and $\omega$ is called an $n$-ary function symbol). 
Then for each $n \in \N_0$, $\Omega_n \subseteq \Omega$ will denote the 
subset of all $n$-ary function symbols from $\Omega$.}

Given a set $\Omega$, an \emph{algebra of type $\Omega$} (shortly, an \emph{$\Omega$-algebra}) 
is a pair $\structA = (A, \Omega)$ where for each $\omega \in \Omega$ with arity $n$, 
there is an $n$-ary operation $\map {f_\omega} {A^n} A$. 

A \emph{\bfseries sup-algebra of type $\Omega$} (shortly, a \emph{sup-algebra}) 
is a triple $\structA = (A,\sup,\Omega)$ where $(A,\sup)$ is a sup-lattice, 
$(A,\Omega)$ is an $\Omega$-algebra, and each operation 
$\omega$ is join-preserving in any component, that is,
$$
\begin{array}{l}\omega \left(a_1,\dots,a_{j-1},\sup B,a_{j+1},\dots,a_n \right) %
 = \\
 \sup \{ \omega(a_1,\dots,a_{j-1},b,a_{j+1},\dots,a_n) \mid b \in B \}
 \end{array}
$$
\noindent{}for any $n \in \N$, $\omega \in \Omega_n$, $j \in \{1,\dots,n\}$, 
$a_1,\dots,a_n \in A$, and $B \subseteq A$. 

A join-preserving mapping $\map \phi A B$ from a sup-algebra $(A,\sup,\Omega)$ 
to a sup-algebra $(B,\sup,\Omega)$ is called a \emph{\bfseries sup-algebra homomorphism} if
$$ \omega_B(\phi(a_1), \dots, \phi(a_n)) = \phi(\omega_A(a_1,\dots,a_n)) $$
\noindent{}for any $n \in \N$, $\omega \in \Omega_n$, and $a_1,\dots,a_n \in A$, and
$$ \omega_B = \phi(\omega_A) $$
\noindent{}for any $\omega \in \Omega_0$. 

Common instances of sup-algebras include the following 
(operation arities that are evident from context are omitted):

\begin{enumerate}

\item sup-lattices with $\Omega = \emptyset$,

\item {\em\bfseries (commutative) quantales} $Q$ \cite{rosenthal1} 
with $\Omega = \{ \cdot \}$ such that $\cdot$ is an associative (and commutative)
binary operation, and  {\em\bfseries unital quantales} $Q$  with $\Omega = \{ \cdot, 1 \}$ such that 
$1$ is the unit of the  associative binary operation $\cdot$,

\item  {\em\bfseries quantale modules} $A$ \cite{rosenthal1} 
with $\Omega = \{ q* \mid q \in Q \}$ such that 
$Q$ is a quantale such that $(\bigvee S)*a=\bigvee_{s\in S} s*a$ and $p*(q*a)=(p\cdot q)*a$ 
for all $S\subseteq Q$, $p, q\in Q$ and $a\in A$.

\end{enumerate}

For any element $q$ of a quantale $Q$, the unary operation $\map {q \cdot -} Q Q$ 
is join-preserving, therefore it has a (meet-preserving) right adjoint $\map {q \ra -} Q Q$, 
characterized by $q \cdot r \leq s \iff r \leq q \ra s$. Written explicitly, $q \ra s = \sup \{ r \in Q \mid q\cdot r \leq s \}$.

Similarly, there is a right adjoint $\map {q \la -} Q Q$ for $- \cdot q$, 
characterized by $r \cdot q \leq s \iff r \leq q \la s$, and satisfying 
$q \la s = \sup \{ r \in Q \mid r\cdot q \leq s \}$. If $Q$ is 
commutative, the operations $\ra$ and $\la$ clearly coincide, and we will keep denoting them $\ra$.

Note that the real unit interval $[0,1]$ with standard partial order and multiplication of reals is a commutative quantale.

\subsection{$Q$-sup-lattices}

By a \emph{base quantale} we mean a unital commutative quantale $Q$. 
The base quantale is the structure in which $Q$-orders and $Q$-subsets 
are to be evaluated. For developing the theory in the rest of this paper, let 
$Q$ be an arbitrary base quantale that remains fixed from now on. Note 
that we do not require the multiplicative unit $1$ of the base quantale to 
be its greatest element $\top$.

Let $X$ be a set. A mapping $\map e {X \times X} Q$ is called a 
\emph{\bfseries$Q$-order} if for any $x,y,z \in X$ the following are satisfied:

\begin{enumerate}

\item $e(x,x) \geq 1$ (reflexivity),

\item $e(x,y) \cdot e(y,z) \leq e(x,z)$ (transitivity),

\item if $e(x,y) \geq 1$ and $e(y,x) \geq 1$, then $x = y$ (antisymmetry).

\end{enumerate}

The pair $(X, e)$ is then called a \emph{\bfseries $Q$-ordered set}. For 
a $Q$-order $e$ on $X$, the relation $\leq_e$ defined as 
$x \leq_e y \iff e(x,y) \geq 1$ is a partial order in the usual sense. This 
means that any $Q$-ordered set can be viewed as an ordinary poset satisfying additional properties.

Vice versa, for a partial order $\leq$ on a set $X$ we can define a $Q$-order $e_\leq$ by

$$ e_\leq(x,y) = \begin{cases} 1, & \mbox{if } x \leq y, \\ 0, & \mbox{otherwise}. \end{cases} $$
A \emph{\bfseries$Q$-subset} of a set $X$ is an element of the set $Q^X$.

For $Q$-subsets $M, N$ of a set $X$, we define the \emph{subsethood degree} of $M$ in $N$ as
$$sub_X(M,N) = \bigwedge_{x \in X} (M(x) \to N(x)).$$
In particular, $(Q^X, sub_X)$ is a $Q$-ordered set. 

Let $M$ be a $Q$-subset of a $Q$-ordered set $(X,e)$. An element 
$s$ of $X$ is called a \emph{\bfseries$Q$-join} of $M$, denoted $\bigsqcup M$ if:

\begin{enumerate}
\item $M(x) \leq e(x,s)$ for all $x \in X$, and
\item for all $y \in X$, $\bigwedge_{x \in X}(M(x) \to e(x,y)) \leq e(s,y)$.
\end{enumerate}

If $\bigsqcup M$ exists for any $M \in Q^X$, we call $(X,e)$ 
\emph{\bfseries $Q$-join complete}, or a \emph{\bfseries$Q$-sup-lattice}.

Let $X$ and $Y$ be sets, and $f \colon X \to Y$ be a mapping. 
\emph{\bfseries Zadeh's forward power set operator} for $f$ maps 
$Q$-subsets of $X$ to $Q$-subsets of $Y$ by 
$$ f_Q^\to (M)(y) = \bigvee_{x \in f^{-1}(y)} M(x).$$

Let $(X, e_X)$ and $(Y, e_Y)$ be $Q$-ordered sets. We say that 
a mapping $f \colon X \to Y$ is \emph{\bfseries $Q$-join-preserving} if for 
any $Q$-subset $M$ of $X$ such that $\bigsqcup M$ exists, $\bigsqcup\nolimits_Y f_Q^\to (M)$ exists and
$$ f\left(\bigsqcup\nolimits_X M\right) = \bigsqcup\nolimits_Y f_Q^\to (M).$$

 It is known that  the category 
$Q\mhyphen\cat{Sup}$ of $Q$-sup-lattices and  $Q$-join-preserving mappings 
is isomorphic to  the category $Q\mhyphen\cat{Mod}$ of  $Q$-modules 
(see \cite{solovyov-qa,stubbe-tcc}). We have functors $F$ and $Q$ such that 
\begin{enumerate}
\item $F \colon Q\mhyphen\cat{Mod} \to Q\mhyphen\cat{Sup}$, given a $Q$-module $A$: \newline
$e(a,b) = a \to_Q b$, and $\Qsup M = \sup_{a \in A}(M(a) * a)$
\item $G \colon Q\mhyphen\cat{Sup} \to Q\mhyphen\cat{Mod}$, given a $Q$-sup-lattice $A$: \newline
$a \leq b \iff 1 \leq e(a,b)$, $\sup S = \Qsup M_S^1$, $q * a = \Qsup M_a^q$ where $M_S^q(a) = \begin{cases} q & \text{if } x \in S, \\ \bot & \text{otherwise}.\end{cases}$
\item A mapping that is a morphism in either category, becomes a morphism in the other one as well.
\item $G \circ F = 1_{Q\mhyphen\cat{Mod}}$ and 
$F \circ G = 1_{Q\mhyphen\cat{Sup}}$, i.e., 
the categories $Q\mhyphen\cat{Mod}$ and  $Q\mhyphen\cat{Sup}$ are isomorphic.
\end{enumerate}
Hence results on quantale modules can be directly transferred to $Q$-sup-lattices and conversely. 
We will speak about {\em\bfseries Solovyov's isomorphism}.

\subsection{$Q$-sup-algebras and $Q$-module-algebras} 

A \emph{\bfseries $Q$-sup-algebra of type $\Omega$} (shortly, a \emph{\bfseries $Q$-sup-algebra}) is 
a triple $\structA = (A, \Qsup, \Omega)$ where $(A,\Qsup)$ is a $Q$-sup-lattice, $(A,\Omega)$ is 
an $\Omega$-algebra, and each operation $\omega$ is $Q$-join-preserving in any component, that is,
$$
\begin{array}{l}
\omega\left( a_1,\dots,a_{j-1},\Qsup M,a_{j+1},\dots,a_n \right) %
=\\
 \Qsup \omega(a_1,\dots,a_{j-1},-,a_{j+1},\dots,a_n)_Q^\to (M)
 \end{array}
$$

\noindent{}for any $n \in \N$, $\omega \in \Omega_n$, 
$j \in \{1,\dots,n\}$, $a_1,\dots,a_n \in A$, and $M \in Q^A$.

Let $(A, \Qsup_A, \Omega)$ and $(B, \Qsup_B, \Omega)$ be $Q$-sup algebras, 
and $\map \phi A B$ be a $Q$-join-preserving mapping and a sup-algebra homomorphism. Then 
$\phi$ is called a \emph{\bfseries $Q$-sup-algebra homomorphism}. 

As instances of $Q$-sup-algebras, we may typically encounter the $Q$-counterparts of those from examples of sup-algebras:

\begin{enumerate}

\item $Q$-sup-lattices ($\Omega = \emptyset$),

\item {\em\bfseries$Q$-quantales} ($\Omega = \{ \cdot \}$) such $\cdot$ is an 
associative binary operation.  $Q$-quantales correspond via Solovyov's isomorphism to 
quantale algebras \cite{solovyov-qa}. 
\end{enumerate}

By \emph{\bfseries $Q$-module-algebra of type $\Omega$} (shortly, \emph{$Q$-module-algebra}) 
we will denote the structure $\structA = (A,\sup,*,\Omega)$ where $(A,\sup,*)$ 
is a $Q$-module, $(A,\Omega)$ is an $\Omega$-algebra, and 
each operation $\omega$ is a $Q$-module homomorphism in any component, that is,
\begin{align*}
& \omega \left(a_1,\dots,a_{j-1},\sup B,a_{j+1},\dots,a_n \right) \\
& \qquad = \sup \{ \omega(a_1,\dots,a_{j-1},b,a_{j+1},\dots,a_n) \mid b \in B \}, \\
& \omega \left(a_1,\dots,a_{j-1},q * b,a_{j+1},\dots,a_n \right) \\
& \qquad =  q * \omega(a_1,\dots,a_{j-1},b,a_{j+1},\dots,a_n)
\end{align*}
for any $n \in \N$, $\omega \in \Omega_n$, $j \in \{1,\dots,n\}$, $a_1,\dots,a_n,b \in A$, and $B \subseteq A$.

A mapping $\map \phi A B$ from a $Q$-module-algebra $(A,\sup,*,\Omega)$ to 
a $Q$-module-algebra $(B,\sup,*,\Omega)$ is called a 
\emph{\bfseries $Q$-module-algebra homomorphism} if it is 
both a $Q$-module homomorphism and an $\Omega$-algebra homomorphism.



For a given quantale $Q$ and a type $\Omega$, let $\catQSupOAlg$ 
denote the category of $Q$-sup-algebras of type $\Omega$ with $Q$-sup-algebra 
homomorphisms, and $\catQModOAlg$ the category of 
$Q$-module-algebras of type $\Omega$ with $Q$-module-algebra homomorphisms. 
From \cite[Theorem  3.3.15.]{slesinger} we know that 
$\catQSupOAlg$ and $\catQModOAlg$ are 
isomorphic via Solovyov's isomorphism.

\section{Some categorical properties of $Q$-sup-algebras} 
\label{sectimpprop}

In this section we establish some categorical properties of the 
category $\catQSupOAlg$ needed in
the sequel. We begin with a construction of a $Q$-sup-algebra of type 
$\Omega$ from 
an $\Omega$-algebra to obtain an adjoint situation. Using this result 
we prove that $\catQSupOAlg$ is a monadic construct 
(see \cite{adamek}).

As shown e.g. in \cite[Theorem~2.2.43]{slesinger}, 
$Q^X$ is the free $Q$-sup-lattice over a set $X$ (and also a free $Q$-module over $X$). 
But we can state more.

\begin{theorem} \label{freecon} Any $\Omega$-algebra $A$ gives 
rise to a $Q$-sup-algebra $Q^A$ with operations defined by
$$\begin{array}{l}\omega_{Q^A}(A_1,\dots,A_n)(a) =\\
 \sup_{\omega_A(a_1,\dots,a_n) = a} A_1(a_1)\cdot \ldots\cdot A_n(a_n),
 \end{array}$$
given $n \in \N$, $a_1,\dots,a_n \in A$, $A_1,\dots,A_n \in Q^A$, $\omega \in \Omega_n$.
\end{theorem}
\begin{proof} By Solovyov's isomorphism it is 
enough to check that  $Q^A$  is a $Q$-module-algebra of type $\Omega$. 
We have 
$$\begin{array}{l}
\omega \left(A_1,\dots,A_{j-1},\sup \Gamma,A_{j+1},\dots,A_n \right)(a) =\\
\sup_{\omega_A(a_1,\dots,a_n) = a} %
A_1(a_1)\cdot\ldots\cdot A_{j-1}(a_{j-1})\cdot\\ 
\sup \Gamma(a_j)\cdot A_{j+1}(a_{j+1})\cdot \ldots\cdot A_n(a_n)=\\
 \sup \{ \sup_{\omega_A(a_1,\dots,a_n) = a} %
A_1(a_1)\cdot\ldots\cdot A_{j-1}(a_{j-1})\cdot\\ 
B(a_j)\cdot A_{j+1}(a_{j+1})\cdot \ldots\cdot A_n(a_n)) \mid B \in \Gamma \}=\\
 \sup \{ \omega(A_1,\dots,A_{j-1},B,A_{j+1},\dots,A_n) \mid B \in \Gamma \}(a)\\
  \end{array}$$
  and
$$\begin{array}{l}
\omega \left(A_1,\dots,A_{j-1},q * B,a_{j+1},\dots,A_n \right)(a)= \\
\sup_{\omega_A(a_1,\dots,a_n) = a} %
A_1(a_1)\cdot\ldots\cdot A_{j-1}(a_{j-1})\cdot \\ 
(q \cdot B(a_j))\cdot A_{j+1}(a_{j+1})\cdot \ldots\cdot A_n(a_n)=\\
\sup_{\omega_A(a_1,\dots,a_n) = a} %
q \cdot (A_1(a_1)\cdot \ldots\cdot A_{j-1}(a_{j-1})\cdot \\ 
B(a_j)\cdot A_{j+1}(a_{j+1})\cdot \ldots\cdot A_n(a_n))=\\
q \cdot \sup_{\omega_A(a_1,\dots,a_n) = a} %
 A_1(a_1)\cdot\ldots\cdot A_{j-1}(a_{j-1})\cdot \\ 
B(a_j)\cdot A_{j+1}(a_{j+1})\cdot \ldots\cdot A_n(a_n)=\\
q * \omega(A_1,\dots,A_{j-1},B,A_{j+1},\dots,A_n)(a)
\end{array}$$
for any $a\in A$, $n \in \N$, $\omega \in \Omega_n$, $j \in \{1,\dots,n\}$, 
$a_1,\dots,a_n \in A$, $A_1,\dots,A_n,B \in Q^A$,
$\Gamma \subseteq Q^A$, and $q \in Q$.
\end{proof}

Note  that there exists the following commutative triangle of 
  the obvious forgetful functors (notice that $\cat{Set}$ is the category
of sets and mappings, and $\catOAlg$ is the category of 
algebras of type $\Omega$ and their homomorphisms): 
\begin{center}
\begin{tikzcd}
\mbox{$\catQSupOAlg$}%
 \arrow{dr}[swap]{V} \arrow{rr}{U} & & \mbox{{\bf Set}} \\
& \mbox{$\catOAlg$} \arrow{ur}[swap]{W} &
\end{tikzcd}
\end{center}

\begin{theorem} The forgetful functor $V\colon \catQSupOAlg \to \catOAlg$
has a left adjoint $F_{Q}$.
\end{theorem}
\begin{proof} Let $A$ be an $\Omega$-algebra. Let us show that $Q^A$ is a free 
$Q$-module-algebra over $A$. For every $a\in A$ there exists a map $\alpha_a\in Q^A$ defined by 
$$ \alpha_a(b) = \begin{cases} 1, & \mbox{if } a=b, \\ \bot, & \mbox{otherwise}. \end{cases} $$
As in \cite[Theorem 3.1]{solovyov-qa} for quantale algebras, we obtain 
a $\Omega$-algebra homomorphism $\eta_A\colon A \to VQ^{A}$ defined by $\eta_A(a)=\alpha_a$. 
Namely, $\eta_A(\omega \left(a_1,\dots,a_n \right))(a)=%
\alpha_{\omega \left(a_1,\dots,a_n \right)}(a)=1$ if and only if
$\omega \left(a_1,\dots,a_n \right)=a$ (otherwise it is $\bot$), and  
$$\begin{array}{l}
\omega \left(\eta_A(a_1),\dots,\eta_A(a_n) \right)(a)=\\%
\sup_{\omega_A(b_1,\dots,b_n) = a} \eta_A(b_1)\cdot \ldots\cdot \eta_A(b_n)=1
\end{array}$$
if and only if $a_j=b_j$ for all $j \in \{1,\dots,n\}$ and $\omega_A(b_1,\dots,b_n) = a$ (otherwise it is $\bot$).  
Hence we obtain
 $$\eta_A(\omega \left(a_1,\dots,a_n \right))=\omega \left(\eta_A(a_1),\dots,\eta_A(a_n) \right).$$
 
 It is easy to show that for every homomorphism $f\colon A\to VB$ in $\catOAlg$ 
 there exists a unique homomorphism $\overline{f}\colon Q^A\to B$ in 
 $\catQSupOAlg$ (given by 
 $\overline{f}(\alpha)=\bigvee_{a\in A} \alpha(a)*f(a)$, where $*$ is the module action on $B$ 
 given by Solovyov's isomorphism) 
 such that the triangle 
 \begin{center}
\begin{tikzcd}
A \arrow{dr}[swap]{f} \arrow{rr}{\eta_{A}} & & VQ^{A} \arrow[dashed]{dl}{V\overline{f}} \\
& VB &
\end{tikzcd}
\end{center}
commutes. We will only check that 
$$\overline{f}(\omega(\alpha_1, \dots, \alpha_n))=%
\omega_B(\overline{f}(\alpha_1), \dots, \overline{f}(\alpha_n)).$$
Let us compute
$$\begin{array}{@{}l}
\overline{f}(\omega(\alpha_1, \dots, \alpha_n))=%
\bigvee_{a\in A} \omega(\alpha_1, \dots)(a)*f(a)\\
=\bigvee_{a\in A} \sup_{\omega_A(a_1,\dots,a_n) = a} 
\alpha_1(a_1)\cdot \ldots\cdot \alpha_n(a_n)*f(a)\\
=\bigvee 
\{(\alpha_1(a_1)\cdot \ldots\cdot \alpha_n(a_n))*f(\omega_A(a_1,\dots,a_n))\mid \\
\phantom{= }\ {a_1,\dots,a_n\in A}\}=\\
\bigvee \{(\alpha_1(a_1)\cdot \ldots\cdot \alpha_n(a_n))*\omega_B(f(a_1),\dots))\mid \\
\phantom{= }\ {a_1,\dots,a_n\in A}\}=\\
\omega_B(\bigvee \{\alpha_1(a_1)* f(a_1) \mid a_1\in A\},\dots))= \\
\omega_B(\overline{f}(\alpha_1), \dots, \overline{f}(\alpha_n)).
\end{array}$$
The remaining properties of  $\overline{f}$ follow by the same considerations as in \cite[Theorem 3.1]{solovyov-qa}.
\end{proof}

\begin{remark}Note  that similarly as in  \cite[Remark 3.2]{solovyov-qa} we obtain 
an adjoint situation $(\eta,\epsilon)\colon F_{Q} \dashv V \colon \catQSupOAlg \to 
\catOAlg$, where $F_Q(A)=$ $Q^{A}$  for every $\Omega$-algebra $A$ of type $\Omega$ and 
$\epsilon_B\colon F_{Q}VB\to B$ is given by $\epsilon_B(\alpha)=\bigvee_{b\in B}\alpha(b)*b=\bigsqcup \alpha$ 
for every $Q$-sup-algebra $B$.
\end{remark}

Since the functor $W\colon \catOAlg \to \cat{Set}$ has a left adjoint (see \cite{adamek}), we obtain 
the following. 

\begin{corollary}
$U\colon \catQSupOAlg \to \cat{Set}$ has a left adjoint.
\end{corollary}

By the same categorical arguments as in   \cite{solovyov-qa} we obtain the following theorem and corollary. 

\begin{theorem}
 The category $\catQSupOAlg$ of $Q$-sup-algebras of type $\Omega$ is a monadic construct.
\end{theorem}

\begin{corollary}
The category $\catQSupOAlg$ 
is complete, cocomplete, wellpowered, 
extremally co-wellpowered, and has regular factorizations. Moreover, monomorphisms 
are precisely those morphisms that are injective functions.
\end{corollary}

\section{$Q$-nuclei in $\catQModOAlg$ and their properties}
\label{nucleiprop}

In this section   we introduce the notion of a $Q$-MA-nucleus for $Q$-module-algebras and present its 
main properties. 

\begin{defn}\label{qnucdef} \rm Let $A$ be a $Q$-module-algebra of type $\Omega$. A 
{\em\bfseries $Q$-module-algebra nucleus on} $A$ (shortly {\bfseries \em $Q$-MA-nucleus} is a
map $j\colon A\to A$ such that 
for any $n \in \N$, $\omega \in \Omega_n$, 
$a, b, a_1,\dots,a_n\in A$, and $q \in Q$:
\begin{enumerate}[(i)]
\item $a \leq b$ implies $j(a) \leq j(b)$;
\item $a \leq j(a)$;
\item  $j \circ j(a) \leq j(a)$;
\item $\omega(j(a_1), \ldots, j(a_n)) \leq  j(\omega(a_1, \ldots, a_n))$;
\item $q * j(a) \leq j(q * a)$.
\end{enumerate}
\end{defn}

Note that by Solovyov's isomorphism $Q$-MA-nuclei correspond to $Q$-ordered algebra nuclei introduced 
in \cite{slesinger}. 

By the same  straightforward computations as in  
\cite[Proposition 4.2 and Corollary 4.3]{solovyov-qa} 
or \cite[Proposition 3.3.6 and Proposition 3.3.8]{slesinger} we obtain the following. 

\begin{proposition} Let $A$ be a $Q$-module-algebra of type $\Omega$ and  $j$ be a $Q$-MA-nucleus on $A$. 
For  any $n \in \N$, $\omega \in \Omega_n$, 
$a, a_1,\dots,a_n\in A$, $S\subseteq A$, and $q \in Q$ we have:
\begin{enumerate}[{\rm (i)}]
\item $(j \circ j)(a) = j(a)$;
\item $j(\bigvee S) = j(\bigvee j(S))$;
\item $ j(\omega(a_1, \ldots, a_n)) = j(\omega(j(a_1), \ldots, j(a_n)))$;
\item $j(q * a) = j(q * j(a))$.
\end{enumerate}
\end{proposition}

\begin{corollary}  Let $A$ be a $Q$-module-algebra of type $\Omega$ and  $j$ be a $Q$-MA-nucleus on $A$.  
 Define $A_j = \{a \in A \mid j(a) = a\}$. Then $A_j = j(A)$ and, 
 moreover, $A_j$ is a $Q$-module-algebra of type $\Omega$  with the
following structure:
\begin{enumerate}[{\rm (i)}]
\item $\bigvee_{A_j} S= j(\bigvee S)$ for every $S \subseteq A_j$;
\item  $\omega_{A_j}(a_1, \ldots, a_n) = j(\omega(a_1, \ldots, a_n))$ for every 
$n \in \N$, $\omega \in \Omega_n$,  $a_1,\dots,a_n\in A_j$;
\item $q *_{A_j} a = j(q * a)$ for every $a \in A_j$ and $q \in Q$.
\end{enumerate}
\end{corollary}

\section{Representation theorem for $Q$-sup-algebras}

Now  we are ready to show a representation theorem for $Q$-sup-algebras. 

Given a  $Q$-module $A$, every $a \in A$ gives rise to the 
adjunction  (in ordered sets)
\begin{tikzcd}
A \arrow[ shift left]{r}{a \twoheadrightarrow \cdot} &  Q \arrow[ shift left]{l}{\cdot * a} 
\end{tikzcd}
 where
$a  \twoheadrightarrow b =  \bigvee\{q \in Q \mid q * a \leq b\}$. 

If moreover $A$ is a $Q$-sup-algebra we use  this adjunction  to 
construct a nucleus on the (free) $Q$-sup-algebra $Q^{V A}$.

\begin{proposition}  Let A be a $Q$-sup-algebra of type $\Omega$. 
There exists a $Q$-MA-nucleus $j_A$ on $Q^{VA}$ 
defined by $j_A(\alpha)(a) = a  \twoheadrightarrow \epsilon_A(\alpha)$.
\end{proposition}
\begin{proof}  It is enough to check the conditions of Definition \ref{qnucdef}.
Conditions (i), (ii), (iii) and (v) follow by the same consideration as in 
\cite[Proposition 5.1]{solovyov-qa}. Let us check  condition (iv).

Let  $n \in \N$, $\omega \in \Omega_n$, 
$a\in A$,  $\alpha_1,\dots,\alpha_n\in Q^{VA}$. We want that  
$\omega(j_A(\alpha_1), \ldots, j_A(\alpha_n))(a) 
\leq  j_A(\omega(\alpha_1, \ldots, \alpha_n))(a)=%
a  \twoheadrightarrow \epsilon_A(\omega(\alpha_1, \ldots, \alpha_n))$. 
We compute 
$$\begin{array}{l}
\omega(j_A(\alpha_1), \ldots, j_A(\alpha_n))(a) * a =\\
\sup_{\omega_A(a_1,\dots,a_n) = a} j_A(\alpha_1)(a_1)\cdot \ldots\cdot j_A(\alpha_n)(a_n)* a=\\
\sup_{\omega_A(a_1,\dots,a_n) = a} \omega_A((a_1\twoheadrightarrow \epsilon_A(\alpha_1)) *a_1,\dots)\leq \\
\omega_A(\epsilon_A(\alpha_1), \dots,  \epsilon_A(\alpha_n))\leq 
\epsilon_A(\omega(\alpha_1, \dots,  \alpha_n)).
\end{array}
$$
The last inequality is valid because $\epsilon_A$ is a homomorphism of $Q$-sup-algebras.
\end{proof}

As in \cite{solovyov-qa} we introduce, for any 
element $a\in A$ of a $Q$-sup-algebra $A$, a map 
$\beta_a\in Q^{VA}$ defined by $\beta_a(x)=x\twoheadrightarrow a$. 
By the same arguments as in \cite[Lemma 5.2]{solovyov-qa} we obtain the following. 

\begin{lemma}  Let $A$ be a $Q$-sup-algebra of type $\Omega$. For every $a\in A$:
$$ \text{(i)\ } \epsilon_A(\beta_a)=a, \qquad \text{(ii)\ }\beta_a\in (Q^{VA})_{j_a}.$$
\end{lemma}

Using the above lemma, we can conclude with our main theorem. 

\begin{theorem} \label{reprtheor} {\normalfont\bfseries (Representation Theorem).} Let 
$A$ be a  $Q$-sup-algebra of type $\Omega$. The map
$\rho_A\colon A\to (Q^{V A})_{j_A}$ defined by 
$\rho_{A}(a) = \beta_a$ is an isomorphism of $Q$-sup-algebras.
\end{theorem}
\begin{proof}
By mimicking the proof of  \cite[Theorem 5.3]{solovyov-qa} we get that 
$\rho_{A}$ is a bijective $Q$-module homomorphism, i.e., it is a bijective homomorphism 
of $Q$-sup-lattices. Let us prove that 
$\omega_{(Q^{V A})_{j_A}}(\rho_{A}(a_1), \dots, \rho_{A}(a_n))=%
\rho_{A}(\omega_A(a_1, \dots,$ $ a_n))$ 
for all $n \in \N$, $\omega \in \Omega_n$, and $a_1,\dots,a_n\in A$. We compute 
$$\begin{array}{l}
\omega_{(Q^{V A})_{j_A}}(\rho_{A}(a_1), \dots, \rho_{A}(a_n))(c)=\\
j_A(\omega(\rho_{A}(a_1), \dots, \rho_{A}(a_n)))(c)=\\
c \twoheadrightarrow \epsilon_A(\omega(\rho_{A}(a_1), \dots, \rho_{A}(a_n)))=\\
c \twoheadrightarrow \omega_A(\epsilon_A(\rho_{A}(a_1)), \dots, \epsilon_A(\rho_{A}(a_n))=\\
c \twoheadrightarrow \omega_A(a_1, \dots, a_n)=
\rho_{A}(\omega_A(a_1, \dots, a_n))(c).
\end{array}
$$
Hence $\rho_{A}$ is an isomorphism of $Q$-sup-algebras.
\end{proof}

\begin{remark}
Note that in case of $Q={\mathbf 2}$ our Theorem follows from \cite[Theorem 2.2.30]{resende}. 
As noted by Solovyov in \cite{solovyov-qa} $\beta_a$ corresponds to the lower set $\da a$ and 
$\epsilon_A\colon {\mathbf 2}^{A}\to A$ is the join operation on $A$.
\end{remark}

\section*{Acknowledgements}  
This is a pre-print of an article published as \newline 
J. Paseka, R.~{\v S}lesinger, A representation theorem for quantale valued sup-algebras, 
in: Proceedings of the 48th IEEE International Symposium on Multiple-Valued Logic, Springer, (2018), 91--96, 
doi: 10.1109/ISMVL.2018.00024.
The final authenticated version of the article is available online at: 
https://ieeexplore.ieee.org/abstract/\-document/8416927.


Both authors acknowledge the support by the bilateral project 
New Perspectives on Residuated Posets  financed by  
Austrian Science Fund (FWF): project I 1923-N25, 
and the Czech Science Foundation (GA\v CR): project 15-34697L.

\end{document}